\documentclass{amsart}
\usepackage[ocgcolorlinks,linktoc=all]{hyperref}
\hypersetup{citecolor=blue,linkcolor=red}
\newtheorem{theorem}{Theorem}[section]
\newtheorem*{thmmain}{Theorem}
\newtheorem{lemma}[theorem]{Lemma}
\newtheorem{proposition}[theorem]{Proposition}
\newtheorem*{propmain}{Proposition}
\newtheorem{corollary}[theorem]{Corollary}
\theoremstyle{definition}

\theoremstyle{remark}

\numberwithin{equation}{section}
\begin{document}

\title[Ancient solutions of the centro-affine normal flow]
 {CENTRO-AFFINE NORMAL FLOWS ON CURVES:\\ HARNACK ESTIMATES AND ANCIENT SOLUTIONS
}

\author[M.N. Ivaki]{Mohammad N. Ivaki}
\address{Institut f\"{u}r Diskrete Mathematik und Geometrie, Technische Universit\"{a}t Wien,
Wiedner Hauptstr. 8--10, 1040 Wien, Austria}
\curraddr{}
\email{mohammad.ivaki@tuwien.ac.at}
\date{}

\dedicatory{}

\begin{abstract}
We prove that the only compact, origin-symmetric, strictly convex ancient solutions of the planar $p$ centro-affine normal flows are contracting origin-centered ellipses.
\end{abstract}

\maketitle

\section{Introduction}
The setting of this paper is the two-dimensional Euclidean space, $\mathbb{R}^2.$ A compact convex subset of $\mathbb{R}^2$ with non-empty interior is called a \emph{convex body}. The set of smooth, strictly convex bodies in $\mathbb{R}^2$ is denoted by $\mathcal{K}$. Write $\mathcal{K}_{0}$ for the set of smooth, strictly convex bodies whose interiors contain the origin of the plane.

Let $K$ be a smooth, strictly convex body $\mathbb{R}^2$ and let $X_K:\partial K\to\mathbb{R}^2$
be a smooth embedding of $\partial K$, the boundary of $K$. Write $\mathbb{S}^1$ for the unit circle and write $\nu:\partial K\to \mathbb{S}^1$ for the Gauss map of $\partial K$. That is, at each point $x\in\partial K$, $\nu(x)$ is the unit outwards normal at $x$.
The support function of $K\in\mathcal{K}_0 $ as a function on the unit circle is defined by
$s(z):= \langle X(\nu^{-1}(z)), z \rangle,$
for each $z\in\mathbb{S}^1$.
We denote the curvature of $\partial K$ by $\kappa$ which as a function on $\partial K$ is related to the support function by
 \[\frac{1}{\kappa(\nu^{-1}(z))}:=\mathfrak{r}(z)=\frac{\partial^2}{\partial \theta^2}s(z)+s(z).\]
Here and afterwards, we identify $z=(\cos \theta, \sin \theta)$ with $\theta$. The function $\mathfrak{r}$ is called the radius of curvature. The affine support function of $K$ is defined by $\sigma:\partial K\to \mathbb{R}$ and $\sigma(x):=s(\nu(x))\mathfrak{r}^{1/3}(\nu(x)).$ The affine support function is invariant under the group of special linear transformations, $SL(2)$, and it plays a basic role in our argument.

Let $K\in \mathcal{K}_0$. A family of convex bodies $\{K_t\}_t\subset\mathcal{K}_0$ given by the smooth map $X:\partial K\times[0,T)\to \mathbb{R}^2$ is said to be a solution to the $p$ centro-affine normal flow, in short $p$-flow, with the initial data $X_K$, if the following evolution equation is satisfied:
\begin{equation}\label{e: flow0}
 \partial_{t}X(x,t)=-\left(\frac{\kappa(x,t)}{\langle X(x,t), \nu(x,t)\rangle^3}\right)^{\frac{p}{p+2}-\frac{1}{3}}\kappa^{\frac{1}{3}}(x,t)\, \nu(x,t),~~
 X(\cdot,0)=X_K,
\end{equation}
for a fixed $1<p<\infty$. In this equation, $0<T<\infty$ is the maximal time that the solution exists, and $\nu(x,t)$ is the unit normal to the curve $ X(\partial K,t)=\partial K_t$ at $X(x,t).$ This family of flows for $p>1$ was defined by Stancu \cite{S}. The case $p=1$ is the well-known affine normal flow whose asymptotic behavior was investigated by Sapiro and Tannebuam \cite{ST}, and by Andrews in a more general setting \cite{BA5,BA4}: Any convex solution to the affine normal flow, after appropriate rescaling converges to an ellipse in the $\mathcal{C}^{\infty}$ norm. For $p>1$, similar result was obtained with smooth, origin-symmetric, strictly convex initial data by the author and Stancu \cite{Ivaki,IS}. Moreover, ancient solutions of the affine normal flow have been also classified: the only compact, convex ancient solutions of the affine normal flow are contracting ellipsoids. This result in $\mathbb{R}^n$, for $n\ge 3$, was proved by Loftin and Tsui \cite{LT} and in dimension two by S. Chen \cite{Chen}, and also by the author with a different method. We recall that a solution of flow is called an ancient solution if it exists on $(-\infty, T)$. Here we classify compact, origin-symmetric, strictly convex ancient solutions of the planar $p$ centro-affine normal flows:
\begin{thmmain}
The only compact, origin-symmetric, strictly convex ancient solutions of the $p$-flows are contracting origin-centred ellipses.
\end{thmmain}
Throughout this paper, we consider origin-symmetric solutions.
\section{Harnack estimate}
In this section, we follow \cite{BA7} to obtain the Harnack estimates for $p$-flows.
\begin{propmain}\label{prop: Harnack estimate} Under the flow (\ref{e: flow0}) we have
$\partial_{t}\left(s^{1-\frac{3p}{p+2}}\mathfrak{r}^{-\frac{p}{p+2}}t^{\frac{p}{2p+2}}\right)\geq0.$
\end{propmain}
\begin{proof}
For simplicity we set $\alpha=-\frac{p}{p+2}.$
To prove the proposition, using the parabolic maximum principle we prove that the quantity defined by
\begin{equation}\label{def: expression R}
\mathcal{R}:=t\mathcal{P}-\frac{\alpha}{\alpha-1}s^{1+3\alpha}\mathfrak{r}^{\alpha}
\end{equation}
remains negative as long as the flow exists. Here $\mathcal{P}$ is defined as follows
\begin{equation*}
\mathcal{P}:=\partial_{t}\left(-s^{1+3\alpha}\mathfrak{r}^{\alpha}\right).
\end{equation*}
\begin{lemma}\label{lem: basic guass}\cite{Ivaki}
\begin{itemize}
  \item $\displaystyle\partial_{t}s=-s^{1+3\alpha}\mathfrak{r}^{\alpha},$
  \item $\displaystyle\partial_{t}\mathfrak{r}=-\left[\left(s^{1+3\alpha}\mathfrak{r}^{\alpha}\right)_{\theta\theta}
      +s^{1+3\alpha}\mathfrak{r}^{\alpha}\right].$
\end{itemize}
\end{lemma}
Using the evolution equations of $s$ and $\mathfrak{r}$ we find
\begin{align}\label{def: Q}
\mathcal{P}&=(1+3\alpha)s^{1+6\alpha}\mathfrak{r}^{2\alpha}+\alpha s^{1+3\alpha}\mathfrak{r}^{\alpha-1}\left[\left(s^{1+3\alpha}\mathfrak{r}^{\alpha}\right)_{\theta\theta}+s^{1+3\alpha}\mathfrak{r}^{\alpha}\right]
\nonumber\\
&:=(1+3\alpha)s^{1+6\alpha}\mathfrak{r}^{2\alpha}+\alpha s^{1+3\alpha}\mathfrak{r}^{\alpha-1}\mathcal{Q}.
\end{align}
\begin{lemma}\label{lem: ev P}
 We have the following evolution equation for $\mathcal{P}$ as long as the flow exists:
\begin{align*}
\partial_{t}\mathcal{P}&=-\alpha s^{1+3\alpha}\mathfrak{r}^{\alpha-1}\left[\mathcal{P}_{\theta\theta}+\mathcal{P}\right]+
\left[(3\alpha+1)(3\alpha+2)-\frac{(\alpha-1)(3\alpha+1)^2}{\alpha}\right]s^{1+9\alpha}\mathfrak{r}^{3\alpha}\\
&+\left[-3(3\alpha+1)+\frac{2(\alpha-1)(3\alpha+1)}{\alpha}\right]s^{3\alpha}\mathfrak{r}^{\alpha}\mathcal{P}
-\frac{\alpha-1}{\alpha}\frac{\mathcal{P}^2}{s^{1+3\alpha}\mathfrak{r}^{\alpha}}.
\end{align*}
\end{lemma}
\begin{proof}
We repeatedly use the evolution equation of $s$ and $\mathfrak{r}$ given in Lemma \ref{lem: basic guass}.
\begin{align*}
&\partial_{t}\mathcal{P}\\
&=
-(1+3\alpha)(1+6\alpha)s^{1+9\alpha}\mathfrak{r}^{3\alpha}-2\alpha(1+3\alpha)s^{1+6\alpha}\mathfrak{r}^{2\alpha-1}
\left[\left(s^{1+3\alpha}\mathfrak{r}^{\alpha}\right)_{\theta\theta}+s^{1+3\alpha}\mathfrak{r}^{\alpha}\right]\\
&-\alpha(1+3\alpha)s^{1+6\alpha}\mathfrak{r}^{2\alpha-1}
\left[\left(s^{1+3\alpha}\mathfrak{r}^{\alpha}\right)_{\theta\theta}+s^{1+3\alpha}\mathfrak{r}^{\alpha}\right]\\
&-\alpha(\alpha-1)s^{1+3\alpha}\mathfrak{r}^{\alpha-2}
\left[\left(s^{1+3\alpha}\mathfrak{r}^{\alpha}\right)_{\theta\theta}+s^{1+3\alpha}\mathfrak{r}^{\alpha}\right]^2
-\alpha s^{1+3\alpha}\mathfrak{r}^{\alpha-1}\left[\mathcal{P}_{\theta\theta}+\mathcal{P}\right]\\
&=-(1+3\alpha)(1+6\alpha)s^{1+9\alpha}\mathfrak{r}^{3\alpha}-3\alpha(1+3\alpha)s^{1+6\alpha}\mathfrak{r}^{2\alpha-1}\mathcal{Q}
\\&-\alpha(\alpha-1)s^{1+3\alpha}\mathfrak{r}^{\alpha-2}\mathcal{Q}^2
-\alpha s^{1+3\alpha}\mathfrak{r}^{\alpha-1}\left[\mathcal{P}_{\theta\theta}+\mathcal{P}\right].
\end{align*}
By the definition of $\mathcal{Q}$, (\ref{def: Q}), we have
$$
\mathcal{Q}^2=\frac{\mathcal{P}^2}{\alpha^2 s^{2+6\alpha}\mathfrak{r}^{2\alpha-2}}-\frac{2(3\alpha+1)}{\alpha^2}\frac{\mathcal{P}\mathfrak{r}^{2}}{s}+
\frac{(3\alpha+1)^2}{\alpha^2}s^{6\alpha}\mathfrak{r}^{2\alpha+2}$$
and
$$\mathcal{Q}=\frac{\mathcal{P}
-(1+3\alpha)s^{1+6\alpha}\mathfrak{r}^{2\alpha}}{\alpha s^{1+3\alpha}\mathfrak{r}^{\alpha-1}}.
$$
Substituting these expressions into the evolution equation of $\mathcal{P}$ we find that
\begin{align*}
\partial_{t}\mathcal{P}&=-\alpha s^{1+3\alpha}\mathfrak{r}^{\alpha-1}\left[\mathcal{P}_{\theta\theta}+\mathcal{P}\right]+
\left[(3\alpha+1)(3\alpha+2)-\frac{(\alpha-1)(3\alpha+1)^2}{\alpha}\right]s^{1+9\alpha}\mathfrak{r}^{3\alpha}\\
&-3(3\alpha+1)s^{3\alpha}\mathfrak{r}^{\alpha}\mathcal{P}-\frac{\alpha-1}{\alpha}\frac{\mathcal{P}^2}{s^{1+3\alpha}\mathfrak{r}^{\alpha}}
+\frac{2(\alpha-1)(3\alpha+1)}{\alpha}s^{3\alpha}\mathfrak{r}^{\alpha}\mathcal{P}.
\end{align*}
This completes the proof of Lemma \ref{lem: ev P}.
\end{proof}
We now proceed to find the evolution equation of $\mathcal{R}$ which is defined by (\ref{def: expression R}). First notice that
$$-\alpha s^{1+3\alpha}\mathfrak{r}^{\alpha-1}\mathcal{R}_{\theta\theta}=-t\alpha s^{1+3\alpha}\mathfrak{r}^{\alpha-1}\mathcal{P}_{\theta\theta}+
\frac{\alpha^2}{\alpha-1}s^{1+3\alpha}\mathfrak{r}^{\alpha-1}\left(s^{1+3\alpha}\mathfrak{r}^{\alpha}\right)_{\theta\theta}.$$
Therefore, by Lemma \ref{lem: ev P} and identity (\ref{def: Q}) we get
\begin{align*}
&\partial_{t}\mathcal{R}\\&=-t\alpha s^{1+3\alpha}\mathfrak{r}^{\alpha-1}\left[\mathcal{P}_{\theta\theta}+\mathcal{P}\right]+
t\left[(3\alpha+1)(3\alpha+2)-\frac{(\alpha-1)(3\alpha+1)^2}{\alpha}\right]s^{1+9\alpha}\mathfrak{r}^{3\alpha}\\
&+t\left[-3(3\alpha+1)+\frac{2(\alpha-1)(3\alpha+1)}{\alpha}\right]s^{3\alpha}\mathfrak{r}^{\alpha}\mathcal{P}
-t\frac{\alpha-1}{\alpha}\frac{\mathcal{P}^2}{s^{1+3\alpha}\mathfrak{r}^{\alpha}}
+\mathcal{P}+\frac{\alpha}{\alpha-1}\mathcal{P}\\
&-\alpha s^{1+3\alpha}\mathfrak{r}^{\alpha-1}\mathcal{R}_{\theta\theta}+t\alpha s^{1+3\alpha}\mathfrak{r}^{\alpha-1}\mathcal{P}_{\theta\theta}
-\frac{\alpha^2}{\alpha-1}s^{1+3\alpha}\mathfrak{r}^{\alpha-1}\left(s^{1+3\alpha}\mathfrak{r}^{\alpha}\right)_{\theta\theta}\\
&+\frac{\alpha^2}{\alpha-1}s^{1+3\alpha}\mathfrak{r}^{\alpha-1}\left(s^{1+3\alpha}\mathfrak{r}^{\alpha}\right)
-\frac{\alpha^2}{\alpha-1}s^{1+3\alpha}\mathfrak{r}^{\alpha-1}\left(s^{1+3\alpha}\mathfrak{r}^{\alpha}\right)\\
&+\frac{\alpha(3\alpha+1)}{\alpha-1}s^{1+3\alpha}\mathfrak{r}^{\alpha-1}\left(s^{1+6\alpha}\mathfrak{r}^{2\alpha}\right)
-\frac{\alpha(3\alpha+1)}{\alpha-1}s^{1+3\alpha}\mathfrak{r}^{\alpha-1}\left(s^{1+6\alpha}\mathfrak{r}^{2\alpha}\right)\\
&=-\alpha s^{1+3\alpha}\mathfrak{r}^{\alpha-1}\mathcal{R}_{\theta\theta}+
t\left[(3\alpha+1)(3\alpha+2)-\frac{(\alpha-1)(3\alpha+1)^2}{\alpha}\right]s^{1+9\alpha}\mathfrak{r}^{3\alpha}\\
&+t\left[-3(3\alpha+1)+\frac{2(\alpha-1)(3\alpha+1)}{\alpha}\right]s^{3\alpha}\mathfrak{r}^{\alpha}\mathcal{P}-t\frac{\alpha-1}{\alpha}
\frac{\mathcal{P}^2}{s^{1+3\alpha}\mathfrak{r}^{\alpha}}+\mathcal{P}\\
&+\frac{\alpha}{\alpha-1}\mathcal{P}-\frac{\alpha}{\alpha-1}\mathcal{P}-t\alpha s^{1+3\alpha}\mathfrak{r}^{\alpha-1}\mathcal{P}+\frac{\alpha^2}{\alpha-1}s^{2+6\alpha}\mathfrak{r}^{2\alpha-1}
+\frac{\alpha(3\alpha+1)}{\alpha-1}s^{2+9\alpha}\mathfrak{r}^{3\alpha-1}\\
&=-\alpha s^{1+3\alpha}\mathfrak{r}^{\alpha-1}\mathcal{R}_{\theta\theta}+
t\left[(3\alpha+1)(3\alpha+2)-\frac{(\alpha-1)(3\alpha+1)^2}{\alpha}\right]s^{1+9\alpha}\mathfrak{r}^{3\alpha}\\
&+t\left[-3(3\alpha+1)+\frac{2(\alpha-1)(3\alpha+1)}{\alpha}\right]s^{3\alpha}\mathfrak{r}^{\alpha}\mathcal{P}-t\frac{\alpha-1}{\alpha}
\frac{\mathcal{P}^2}{s^{1+3\alpha}\mathfrak{r}^{\alpha}}+\mathcal{P}\\
&-t\alpha s^{1+3\alpha}\mathfrak{r}^{\alpha-1}\mathcal{P}+\frac{\alpha^2}{\alpha-1}s^{2+6\alpha}\mathfrak{r}^{2\alpha-1}
+\frac{\alpha(3\alpha+1)}{\alpha-1}s^{2+9\alpha}\mathfrak{r}^{3\alpha-1}.\\
\end{align*}
In the last expression, using the definition of $\mathcal{R}$, identity (\ref{def: expression R}), we replace $t\mathcal{P}$ by
$\mathcal{R}+\frac{\alpha}{\alpha-1}s^{1+3\alpha}\mathfrak{r}^{\alpha}.$ Therefore, at the point where the maximum of $\mathcal{R}$ is achieved we obtain
\begin{align*}
&\partial_{t}\mathcal{R}\\&\leq \mathcal{R}\left[-\alpha s^{1+3\alpha}\mathfrak{r}^{\alpha-1}-\frac{\alpha-1}{\alpha}\frac{\mathcal{P}}{s^{1+3\alpha}\mathfrak{r}^{\alpha}}
+\left[\frac{2(\alpha-1)(3\alpha+1)}{\alpha}-3(3\alpha+1)\right]s^{3\alpha}\mathfrak{r}^{\alpha}\right]\\
&+\frac{\alpha}{\alpha-1}\left[\frac{2(\alpha-1)(3\alpha+1)}{\alpha}-3(3\alpha+1)\right]s^{2+6\alpha}\mathfrak{r}^{2\alpha}
+\frac{\alpha(3\alpha+1)}{\alpha-1}s^{2+9\alpha}\mathfrak{r}^{3\alpha-1}\\
&+t\left[(3\alpha+1)(3\alpha+2)-\frac{(\alpha-1)(3\alpha+1)^2}{\alpha}\right]s^{1+9\alpha}\mathfrak{r}^{3\alpha}\\
&\leq  \mathcal{R}\left[-\alpha s^{1+3\alpha}\mathfrak{r}^{\alpha-1}-\frac{\alpha-1}{\alpha}\frac{\mathcal{P}}{s^{1+3\alpha}\mathfrak{r}^{\alpha}}
+\left[\frac{2(\alpha-1)(3\alpha+1)}{\alpha}-3(3\alpha+1)\right]s^{3\alpha}\mathfrak{r}^{\alpha}\right].
\end{align*}
To get the last inequality, we used the fact that the terms on the second and third line are negative for $p\geq 1$. Hence, by the parabolic maximum principle and the fact that at the time zero we have $\mathcal{R}\leq0$, we conclude $\mathcal{R}=t\mathcal{P}-\frac{\alpha}{\alpha-1}s^{1+3\alpha}\mathfrak{r}^{\alpha}\leq 0$.
Negativity of $\mathcal{R}$ is equivalent to
$\partial_{t}\ln\left(s^{1+3\alpha}\mathfrak{r}^{\alpha}\right)\geq\frac{\alpha}{1-\alpha}\frac{1}{t}$
for $t>0.$
From this we infer that
$\partial_{t}\left(s^{1+3\alpha}\mathfrak{r}^{\alpha}t^{\frac{\alpha}{\alpha-1}}\right)\geq 0$
for $t>0.$
\end{proof}
\begin{proposition}
Ancient solutions of the flow (\ref{e: flow0}) satisfy
 $\partial_{t}\left(s\left(\frac{1}{\mathfrak{r}s^3}\right)^{\frac{p}{p+2}}\right)\geq 0.$
 \end{proposition}
 \begin{proof}
 By the Harnack estimate every solution of the flow (\ref{e: flow0}) satisfies
 \begin{equation}\label{ie: Harnack Euclidean}
 \partial_{t}\left(s\left(\frac{1}{\mathfrak{r}s^3}\right)^{\frac{p}{p+2}}\right)+\frac{p}{2t(p+1)}\left(s\left(\frac{1}{\mathfrak{r}s^3}\right)^{\frac{p}{p+2}}\right)\geq0.
 \end{equation}
 We let the flow starts from a fixed time $t_0<0$. So the inequality (\ref{ie: Harnack Euclidean}) becomes
  \begin{equation*}
 \partial_{t}\left(s\left(\frac{1}{\mathfrak{r}s^3}\right)^{\frac{p}{p+2}}\right)+\frac{p}{2(t-t_0)(p+1)}\left(s\left(\frac{1}{\mathfrak{r}s^3}\right)^{\frac{p}{p+2}}\right)\geq0.
 \end{equation*}
 Now letting $t_0$ goes to $-\infty$ proves the claim.
\end{proof}
\begin{corollary}\label{cor: harnack estimate all flow}
Every ancient solution of the flow (\ref{e: flow0}) satisfies
$\partial_{t}\left(s\mathfrak{r}^{\frac{1}{3}}\right)\leq 0.$
\end{corollary}
\begin{proof}
The $s(\cdot,t)$ is decreasing on the time interval $(-\infty, 0].$ The claim now follows from the previous proposition.
\end{proof}
\section{Affine differential setting}
We will recall several definitions from affine differential geometry. Let $\gamma:\mathbb{S}^1\to\mathbb{R}^2$ be an embedded strictly convex curve with the curve parameter $\theta$. Define $\mathfrak{g}(\theta):=[\gamma_{\theta},\gamma_{\theta\theta}]^{1/3}$, where for two vectors $u, v$ in $\mathbb{R}^2$, $[u, v]$ denotes the determinant of
the matrix with rows $u$ and $v$. The affine arc-length is defined as
\begin{equation*}
\mathfrak{s}(\theta):=\int_{0}^{\theta}\mathfrak{g}(\alpha)d\alpha.
\end{equation*}
Furthermore, the affine normal vector $\mathfrak{n}$ is given by
$ \mathfrak{n}:=\gamma_{\mathfrak{s}\mathfrak{s}}.$
In the affine coordinate ${\mathfrak{s}}$, there hold $[\gamma_{\mathfrak{s}},\gamma_{\mathfrak{s}\mathfrak{s}}]=1,$ $\sigma=[\gamma,\gamma_{\mathfrak{s}}],$ and $\sigma_{\mathfrak{s}\mathfrak{s}}+\sigma\mu=1,$ where $\mu=[\gamma_{\mathfrak{s}\mathfrak{s}},\gamma_{\mathfrak{s}\mathfrak{s}\mathfrak{s}}]$ is the affine curvature.

We can express the area of $K\in\mathcal{K}$, denoted by $A(K)$, in terms of affine invariant quantities:
$$A(K)=\frac{1}{2}\int_{\partial K}\sigma d\mathfrak{s}.$$
The $p$-affine perimeter of $K\in \mathcal{K}_0$ (for $p=1$ the assumption $K\in \mathcal{K}_0$ is not necessary and we may take $K\in \mathcal{K}$), denoted by $\Omega_p(K)$, is defined as
\[\Omega_p(K):=\int_{\partial K}\sigma^{1-\frac{3p}{p+2}}d\mathfrak{s},\]
\cite{Lutwak2}. We call the quantity $\Omega_p^{2+p}(K)/A^{2-p}(K),$ the $p$-affine isoperimetric ratio and mention that it is invariant under $GL(2).$
Moreover, for $p>1$ the $p$-affine isoperimetric inequality states that if $K$ has its centroid at the origin, then
\begin{equation}\label{ie: affine isoperimetric}
\displaystyle\frac{\Omega_p^{2+p}(K)}{A^{2-p}(K)}\leq 2^{p+2}\pi^{2p}
\end{equation}
and equality cases are obtained only for origin-centered ellipses. In the final section, we will use the $2$-affine isoperimetric inequality .

Let $K\in\mathcal{K}_0$. The polar body of $K$, denoted by $K^{\ast}$, is a convex body in $\mathcal{K}_0$ defined by
\[
K^{\ast} = \{ y \in \mathbb{R}^{2} \mid \langle x , y\rangle \leq 1,\ \forall x
\in K \}.\]

The area of $K^{\ast}$, denoted by $A^{\ast}=A(K^{\ast})$, can be
represented in terms of affine invariant quantities:
$$A^{\ast}=\frac{1}{2}\int_{\partial K}\frac{1}{\sigma^2}d\mathfrak{s}=
\frac{1}{2}\int_{\mathbb{S}^1}\frac{1}{s^2}d\theta.$$

Let $K\in \mathcal{K}_0$. We consider a family of convex bodies $\{K_t\}_t\subset \mathcal{K}$, given by the smooth embeddings $X:\partial K\times[0,T)\to \mathbb{R}^2$, which are
evolving according to (\ref{e: flow0}). Then up to a time-dependant diffeomorphism, $\{K_t\}_t$ evolves according to
\begin{equation}\label{e: affine def of flow}
 \frac{\partial}{\partial t}X:=\sigma^{1-\frac{3p}{p+2}}\mathfrak{n},~~
 X(\cdot,0)=X_{K}(\cdot).
\end{equation}
Therefore, classification of compact, origin-symmetric ancient solutions to (\ref{e: flow0}) is equivalent to the classification of compact, origin-symmetric ancient solutions to
(\ref{e: affine def of flow}). In what follows our reference flow is the evolution equation (\ref{e: affine def of flow}).

Notice that as a family of convex bodies evolve according to the evolution equation (\ref{e: affine def of flow}), in the Gauss parametrization their support functions and radii of curvature evolve according to Lemma \ref{lem: basic guass}. Assume $Q$ and $\bar{Q}$ are two smooth functions $Q:\partial K\times [0,T)\to \mathbb{R}$, $\bar{Q}:\mathbb{S}^1\times[0,T)\to \mathbb{R}$ that are related by
$Q(x,t)=\bar{Q}(\nu(x,t),t)$. It can be easily verified that
\[\partial_{t}\bar{Q}=\partial_{t}Q-Q_{\mathfrak{s}}\left(\sigma^{1-\frac{3p}{p+2}}\right)_{\mathfrak{s}}.\]
In particular, for ancient solutions of (\ref{e: affine def of flow}), in views of Corollary \ref{cor: harnack estimate all flow}, $Q=\sigma$ must satisfy
$
 0\geq
\partial_{t}\sigma-\sigma_{\mathfrak{s}}\left(\sigma^{1-\frac{3p}{p+2}}\right)_{\mathfrak{s}}.
$
The proceeding argument proves the next proposition.
\begin{proposition}\label{cor: harnack estimate all flow1}
Every ancient solution  satisfies
$\partial_{t}\sigma\leq -\left(\frac{3p}{p+2}-1\right)\sigma_{\mathfrak{s}}^2\sigma^{-\frac{3p}{p+2}}.$
\end{proposition}
The next two lemmas were proved in \cite{Ivaki}.
\begin{lemma}\cite[Lemma 3.1]{Ivaki}\label{e: basic ev}
The following evolution equations hold:
\begin{enumerate}
\item $\displaystyle \frac{\partial}{\partial t}\sigma=\sigma^{1-\frac{3p}{p+2}}\left(-\frac{4}{3}+
\left(\frac{p}{p+2}+1\right)\left(1-\frac{3p}{p+2}\right)\frac{\sigma_{\mathfrak{s}}^2}{\sigma}+\frac{p}{p+2}\sigma_{\mathfrak{s}\mathfrak{s}}
\right),$
\item $\displaystyle \frac{d}{dt}A=-\Omega_p.$
\end{enumerate}
\end{lemma}
\begin{lemma}\cite[Section 6]{Ivaki}
The following evolution equation for $\Omega_{l}$ holds for every $l\geq2$ and $p\geq 1$:
\begin{equation}\label{e: eveq general}
\frac{d}{dt}\Omega_l(t)=\frac{2(l-2)}{l+2}\int_{\gamma_t}\sigma^{1-\frac{3p}{p+2}-\frac{3l}{l+2}}d\mathfrak{s}
+\frac{18pl}{(l+2)^2(p+2)}\int_{\gamma_t}\sigma^{-\frac{3p}{p+2}-\frac{3l}{l+2}}\sigma_{\mathfrak{s}}^2d\mathfrak{s},
\end{equation}
where $\gamma_t:=\partial K_t$ is the boundary of $K_t$.
\end{lemma}
\begin{lemma}\label{lem: mon vol prod}\cite{S}
The area product, $A(t)A^{\ast}(t)$, and the $p$-affine isoperimetric ratio are both non-decreasing along (\ref{e: affine def of flow}).
\end{lemma}
Write respectively $\max\limits_{\gamma_t}\sigma$ and $\min\limits_{\gamma_t}\sigma$ for $\sigma_M$ and $\sigma_m$.
\begin{lemma} There is a constant $0<c<\infty$ such that
$\frac{\sigma_{M}}{\sigma_m}\leq c$
on $(-\infty,0].$
\end{lemma}
\begin{proof}
By Corollary \ref{cor: harnack estimate all flow1} and part (1) of Lemma \ref{e: basic ev} we have
 \begin{align}\label{e: fundamental}
-\left(\frac{3p}{p+2}-1\right)\frac{\sigma_{\mathfrak{s}}^2}{\sigma^3}&\geq\frac{\partial_t\sigma}{\sigma^{3-\frac{3p}{p+2}}}\nonumber\\
&=-\frac{4}{3\sigma^2}
+\left(\frac{p}{p+2}+1\right)\left(1-\frac{3p}{p+2}\right)
\frac{\sigma_{\mathfrak{s}}^2}{\sigma^3}
+\frac{p}{p+2}\frac{\sigma_{\mathfrak{s}\mathfrak{s}}}{\sigma^2}.
 \end{align}
 Integrating the inequality (\ref{e: fundamental}) against $d\mathfrak{s}$ we obtain
\begin{align}\label{ie: intergration}
\frac{4}{3}\int_{\gamma_t}\frac{1}{\sigma^2}d\mathfrak{s}&\geq \frac{p}{p+2}\left(2-\frac{3p}{p+2}\right)\int_{\gamma_t}
\frac{\sigma_{\mathfrak{s}}^2}{\sigma^3}d\mathfrak{s}\\
&=\frac{p}{p+2}\left(3-\frac{3p}{p+2}\right)\int_{\gamma_t}
\frac{(\ln \sigma)_{\mathfrak{s}}^2}{\sigma}d\mathfrak{s}\nonumber\\
&\geq \frac{p}{p+2}\left(3-\frac{3p}{p+2}\right)\frac{\left(\int_{\gamma_t}
|(\ln \sigma)_{\mathfrak{s}}|d\mathfrak{s}\right)^2}{\int_{\gamma_t}\sigma d\mathfrak{s}}.\nonumber
\end{align}
Set $d_{p}=\frac{p}{p+2}\left(3-\frac{3p}{p+2}\right)$.
Applying the H\"{o}lder inequality to the left-hand side and the right-hand side of inequality (\ref{ie: intergration}) yields
\begin{equation*}
\left(\int_{\gamma_t}|\left(\ln\sigma\right)_{\mathfrak{s}}|d\mathfrak{s}\right)^2\leq d_p'A^{\ast}(t)A(t),
\end{equation*}
for a new positive constant $d'_p.$ Here we used the identities $\int_{\gamma_t}\frac{1}{\sigma^2}d\mathfrak{s}=2A^{\ast}(t)$ and $\int_{\gamma_t}\sigma d\mathfrak{s}=2A(t).$ Now by Lemma \ref{lem: mon vol prod} we have $A(t)A^{\ast}(t)\leq A(0)A^{\ast}(0).$ This implies that
\begin{equation*}
\left(\ln\frac{\sigma_{M}}{\sigma_m}\right)^2\leq d_p'',
\end{equation*}
for a new positive constant $d''_p.$
Therefore, on $(-\infty,0]$ we find that
\begin{equation}\label{ie: sigma ratio}
\frac{\sigma_{M}}{\sigma_m}\leq c
\end{equation}
for some positive constant $c$.
\end{proof}
Let $\{K_t\}_t$ be a solution of (\ref{e: affine def of flow}). Then the family of convex bodies, $\{\tilde{K}_t\}_t$, defined by
$$\tilde{K}_t:=\sqrt{\frac{\pi}{A(K_t)}}K_t$$
is called a normalized solution to the $p$-flow, equivalently a solution that the area is fixed and is equal to $\pi.$

Furnish every quantity associated with the normalized solution with an over-tilde. For example, the support function, curvature, and the affine support function of $\tilde{K}$ are denoted by $\tilde{s}$, $\tilde{\kappa}$, and $\tilde{\sigma}$, respectively.
\begin{lemma} There is a constant $0<c<\infty$ such that on the time interval $(-\infty,0]$ we have
\begin{equation}\label{ie: uniform para1}
\frac{\tilde{\sigma}_{M}}{\tilde{\sigma}_m}\leq c.
\end{equation}
\end{lemma}
\begin{proof}
The estimate (\ref{ie: sigma ratio}) is scaling invariant, so the same estimate holds for the normalized solution.
\end{proof}
\begin{lemma}\label{lem: controlling derivative of $l$-affine surface area along affine normal flow}
$\Omega_2(t)$ is non-decreasing along the $p$-flow. Moreover, we have
\begin{align*}
\frac{d}{dt}\Omega_2(t)\geq \frac{9p}{4(p+2)}\int_{\gamma_t}\sigma^{-\frac{3p}{p+2}-\frac{3}{2}}\sigma_{\mathfrak{s}}^2d\mathfrak{s}.
\end{align*}
\end{lemma}
\begin{proof}
Use the evolution equation (\ref{e: eveq general}) for $l=2$.
\end{proof}
\begin{corollary}\label{cor: upper bound of invers ratio}
There exists a constant $0<b_{p}<\infty$ such that
\[\frac{1}{\Omega_2^{4}(t)}<b_{p}\]
on $(-\infty, 0].$
\end{corollary}
\begin{proof}
Notice that
$\Omega_2(t)=\left(\int_{\partial\gamma_t}\sigma^{-\frac{1}{2}}d\mathfrak{s}\right)$
is a $GL(2)$ invariant quantity. Therefore, we need only to prove the claim after applying appropriate $SL(2)$ transformations to the normalized solution of the flow.
By the estimate (\ref{ie: uniform para1}) and the facts that $\Omega_2(\tilde{K}_{t})$ is non-decreasing and $A(\tilde{K}_t)=\pi$ we have
\[\frac{c^{\frac{3}{2}}}{2}\tilde{\sigma}_m^{\frac{3}{2}}(t)\tilde{\Omega}_2(0)\geq\frac{1}{2}\tilde{\sigma}_M^{\frac{3}{2}}(t)\tilde{\Omega}_2(0)
\geq\frac{1}{2}\tilde{\sigma}_M^{\frac{3}{2}}(t)\tilde{\Omega}_2(t)\geq
\tilde{ A}(t)=\pi.\]
So we get $\left(\tilde{s}\tilde{\mathfrak{r}}^{1/3}\right)(t)\geq a>0$ on $(-\infty,0]$, for an $a$ independent of $t$. Moreover, as the affine support
function is invariant under $SL(2)$ we may further assume, after applying a length minimizing special linear transformation at each time, that $\tilde{s}(t)<a'<\infty$,
for an $a'$ independent of $t$.
Therefore
\begin{align}\label{e:11}
\frac{\tilde{\Omega}_1^3(t)}{\tilde{A}(t)}=\frac{(\int_{\mathbb{S}^1}\tilde{\mathfrak{r}}^{2/3}d\theta)^3}{\pi}>a''>0,
\end{align}
for an $a''$ independent of $t$.
Now the claim follows from the H\"{o}lder inequality:
\[\left(\int_{\gamma_t}\sigma^{-\frac{1}{2}}d\mathfrak{s}\right)\Omega_1^{\frac{1}{2}}(t)A^{\frac{1}{2}}(t)\geq \int_{\gamma_t}\sigma^{-\frac{1}{2}}d\mathfrak{s}
\int_{\gamma_t}\sigma^{\frac{1}{2}}d\mathfrak{s}
\geq \Omega_1^2(t),\]
so
\[\tilde{\Omega}_2(t)=\Omega_2(t)\geq \left(\frac{\Omega_1^3(t)}{A(t)}\right)^{\frac{1}{2}}=\left(\frac{\tilde{\Omega}_1^3(t)}{\tilde{A}(t)}\right)^{\frac{1}{2}}.\]
\end{proof}
\begin{corollary}\label{cor: liminf idea} As $K_t$ evolve by (\ref{e: affine def of flow}), then the following limit holds:
\begin{equation}
 \liminf_{t\to -\infty}\left(\frac{A(t)}{\Omega_p(t)\Omega_2^{5}(t)}\right)
\int_{\gamma_t}\left(\sigma^{\frac{1}{4}-\frac{3p}{2(p+2)}}\right)_{\mathfrak{s}}^2d\mathfrak{s}=0.
 \label{eq:sup}
 \end{equation}
\end{corollary}
\begin{proof}
Suppose on the contrary that there exists an $\varepsilon>0$ small enough, such that
\[\left(\frac{A(t)}{\Omega_p(t)\Omega_2^{5}(t)}\right)
\int_{\gamma_t}\left(\sigma^{\frac{1}{4}-\frac{3p}{2(p+2)}}\right)_{\mathfrak{s}}^2d\mathfrak{s}\geq \varepsilon\frac{\left(\frac{1}{4}-\frac{3p}{2(p+2)}\right)^2}{\frac{9p}{p+2}}\]
on $(-\infty, -N]$ for $N$ large enough. Then
$
\frac{d}{dt}\frac{1}{\Omega_2^4(t)}
\leq\varepsilon\frac{d}{dt}\ln(A(t)).
$
So by integrating this last inequality against $dt$ and by  Corollary \ref{cor: upper bound of invers ratio} we get
\begin{align*}
0<\frac{1}{\Omega_2^4(-N)}&\leq \frac{1}{\Omega_2^4(t)}+\varepsilon \ln(A(-N))-\varepsilon \ln(A(t))\\
&< b_p+\varepsilon \ln(A(-N))-\varepsilon \ln(A(t)).
\end{align*}
Letting $t\to-\infty$ we reach to a contradiction: $\lim\limits_{t\to-\infty}A(t)=+\infty$, that is, the right-hand side becomes negative for large values of $t.$
\end{proof}
\begin{corollary}\label{cor: limit of sigma}
For a sequence of times $\{t_k\}$ as $t_k$ converge to $-\infty$ we have
\[\lim_{t_k\to-\infty}\tilde{\sigma}(t_k)=1.\]
\end{corollary}
\begin{proof}
Notice that the quantity $\left(\frac{A(t)}{\Omega_p(t)\Omega_2^{5}(t)}\right)
\int_{\gamma_t}\left(\sigma^{\frac{1}{4}-\frac{3p}{2(p+2)}}\right)_{\mathfrak{s}}^2d\mathfrak{s}$ is scaling invariant and $\frac{\tilde{A}(t)}{\tilde{\Omega}_p(t)\tilde{\Omega}_2^{5}(t)}$ is bounded from below (By Lemmas \ref{lem: mon vol prod} and \ref{lem: controlling derivative of $l$-affine surface area along affine normal flow}, $\tilde{\Omega}_p(t)\leq \tilde{\Omega}_p(0)$ and $\tilde{\Omega}_2(t)\leq \tilde{\Omega}_2(0)$.). Thus Corollary \ref{cor: liminf idea} implies that there exists a sequence of times $\{t_k\}_{k\in\mathbb{N}}$, such that $\lim\limits_{k\to\infty}t_k=-\infty$ and
$$\lim\limits_{t_k\to-\infty}\int_{\tilde{\gamma}_{t_k}}
\left(\tilde{\sigma}^{\frac{1}{4}-\frac{3p}{2(p+2)}}\right)^2_{\tilde{\mathfrak{s}}}d\tilde{\mathfrak{s}}=0.$$
On the other hand, by the H\"{o}lder inequality
\begin{align*}
\frac{\left(\tilde{\sigma}_M^{\frac{1}{4}-\frac{3p}{2(p+2)}}(t_k)-\tilde{\sigma}_m^{\frac{1}{4}-\frac{3p}{2(p+2)}}(t_k)\right)^2}
{\tilde{\Omega}_1(t_k)}\leq \int_{\tilde{\gamma}_{t_k}}
\left(\tilde{\sigma}^{\frac{1}{4}-\frac{3p}{2(p+2)}}\right)^2_{\tilde{\mathfrak{s}}}d\tilde{\mathfrak{s}}.
\end{align*}
Moreover, $\tilde{\Omega}_1(t)$ is bounded from above: Indeed $\left(\frac{\tilde{\Omega}_1^3(t)}{\tilde{A}(t)}\right)^{\frac{1}{2}}\leq \tilde{\Omega}_2(t)\leq \tilde{\Omega}_2(0)$. Therefore, we find that
$$\lim\limits_{t_k\to-\infty}\left(\tilde{\sigma}_M^{\frac{1}{4}-\frac{3p}{2(p+2)}}(t_k)-
\tilde{\sigma}_m^{\frac{1}{4}-\frac{3p}{2(p+2)}}(t_k)\right)^2=0.$$
Since $\tilde{\sigma}_m\leq1$ and $\tilde{\sigma}_M\geq1$ (see \cite[Lemma 10]{BA2}) the claim follows.
\end{proof}
\section{Proof of the main Theorem}
\begin{proof}
For each time $t\in (-\infty, T)$, let $T_t\in SL(2)$ be a special linear transformation that the maximal ellipse contained in $T_t\tilde{K}_t$ is a disk. So by John's ellipsoid lemma we have
\[\frac{1}{\sqrt{2}}\leq s_{T_{t}\tilde{K}_{t}}\leq \sqrt{2}.\]
Then by the Blaschke selection theorem, there is a subsequence of times, denoted again by $\{t_k\}$, such that
$\{T_{t_k}\tilde{K}_{t_k}\}$ converges in the Hausdorff distance to an origin-symmetric convex body $\tilde{K}_{-\infty}$, as $t_k\to-\infty$. By Corollary \ref{cor: limit of sigma}, and by the weak convergence
of the Monge-Amp\`{e}re measures, the support function of $\tilde{K}_{-\infty}$ is the generalized solution of the following Monge-Amp\`{e}re equation on $\mathbb{S}^1$:
\[s^3(s_{\theta\theta}+s)=1\]
Therefore, by Lemma 8.1 of Petty \cite{Petty}, $\tilde{K}_{-\infty}$ is an origin-centered ellipse. This in turn implies that $\lim\limits_{t\to -\infty}\tilde{\Omega}_2(t_k)=2\pi.$
On the other hand, by the $2$-affine isoperimetric inequality, (\ref{ie: affine isoperimetric}), and by Lemma \ref{lem: controlling derivative of $l$-affine surface area along affine normal flow}, for $t\in (\infty, 0]$ we have
\[2\pi\geq \tilde{\Omega}_2(t)\geq \lim_{t_k\to -\infty}\tilde{\Omega}_2(t_k)=2\pi.\]
Thus $\frac{d}{dt}\tilde{\Omega}_2(t)\equiv 0$ on $(-\infty, 0].$ Hence, in view of Lemma \ref{lem: controlling derivative of $l$-affine surface area along affine normal flow}, $K_t$ is an origin-centred ellipse for every time $t\in(-\infty, T).$
\end{proof}

\bibliographystyle{amsplain}

\end{document}